\pdfoutput=1
\documentclass[a4paper,10pt]{article}
\usepackage[utf8]{inputenc}
\usepackage{latexsym}
\usepackage{amsthm}
\usepackage{amssymb}
\usepackage{amsmath}
\usepackage[all]{xy}
\theoremstyle{plain}
\newtheorem{thm}{Theorem}[section]
\newtheorem{lemma}[thm]{Lemma}
\newtheorem{prop}[thm]{Proposition}
\newtheorem{cor}[thm]{Corollary}
\newtheorem*{thrm-adm}{Theorem \ref{adm}}
\newtheorem*{thrm-vsp}{Theorem \ref{vsp}}
\newtheorem*{thrm-fg-cen}{Theorem \ref{fg-cen-thm}}
\newtheorem*{thrm-rel}{Theorem \ref{rel-graph}}
\newtheorem*{thrm-hoch}{Theorem \ref{hochschild}}
\theoremstyle{definition}
\newtheorem{defn}[thm]{Definition}

\newtheorem{exmp}[thm]{Example}

\theoremstyle{remark}

\newcommand{\La}{\Lambda}

\newcommand{\HH}{\mathrm{HH}^*}
\newcommand{\Ni}{\mathcal{N}}

\newcommand{\op}{\mathrm{op}}
\newcommand{\tbf}{\textbf}
\newcommand{\ori}{\mathfrak{o}}
\newcommand{\tar}{\mathfrak{t}}
\newcommand{\und}{\text{ --- }}
\newcommand{\ev}{\mathrm{ev}}
\newcommand{\gr}{\mathrm{gr}}

\newcommand{\mrm}{\mathrm}
\title{Centers of partly (anti-)commutative quiver algebras 
and finite generation of the Hochschild cohomology ring}
\author{Elin Gawell and Qimh Richey Xantcha}

\begin{document}

\maketitle

\begin{abstract}
A \emph{partly (anti-)commutative quiver algebra} is a quiver algebra bound by 
an \emph{(anti-)commutativity ideal}, that is, a quadratic ideal generated by 
monomials and (anti-)commutativity relations. 
We give a combinatorial description
of the ideals and the associated \emph{generator graphs}, 
from which one can quickly determine 
if the ideal is admissible or not. 
We describe the center of a partly (anti-)commutative quiver
algebra and state necessary and sufficient conditions for the center to be finitely generated
as a $K$-algebra.
As an application, necessary and sufficient conditions for finite generation 
of the Hochschild cohomology ring modulo 
nilpotent elements for a partly 
\mbox{(anti-)}commutative Koszul quiver algebra are given. 
\end{abstract}

\section{Introduction}

The theory of support varieties of finitely generated modules over a finite-dimensional $K$-algebra $\Lambda$ 
using Hochschild cohomology was introduced by Sna\-sh\-all and Solberg in \cite{snashall2004support}. One essential
property needed to apply their theory is that the Hochschild cohomology ring modulo nilpotent elements, 
$\mathrm{HH}^*(\Lambda)/\mathcal{N}$, is finitely
generated as a $K$-algebra. This was known to be true for several types of algebras such as finite-dimensional selfinjective algebras of 
finite representation type over an algebraically closed field \cite{green2003hochschild}, finite-dimensional monomial
algebras \cite{green2004hochschild}, finite-dimensional algebras of finite global dimension \cite{happel1989hochschild},
any block of a group ring of a finite group \cite{evens1961cohomology}, any block of a finite dimensional
cocommutative Hopf algebra \cite{friedlander1997cohomology} and 
when $\Lambda=KQ/I$ is a finite-dimensional Nakayama
algebra over a field $K$ bound by an admissible ideal $I$ generated by a single relation \cite{snashall2004support}. 
Snashall and Solberg conjectured that $\mathrm{HH}^*(\Lambda)/\mathcal{N}$ is always finitely
generated as a $K$-algebra when $\Lambda$ is a finite dimensional algebra over a field $K$.
Xu found a 
counterexample to this conjecture when the field $K$ has characteristic $2$ \cite{xu2008hochschild}, which was 
later generalized to all characteristics
by Snashall \cite{snashall2008support}. 
Several people have been working on finding the
necessary and sufficient conditions to make $\mathrm{HH}^*(\Lambda)/\mathcal{N}$ a finitely generated algebra over 
$K$, for example Parker and Snashall \cite{parker2011family}, and from this work also more classes
of counterexamples have been found \cite{xu2011more}.

For Koszul quiver algebras we have that $\HH(\La)/\Ni\cong Z_\gr(\La^!)/\Ni_Z$ \cite{green1998koszul} and hence knowledge
of the center of a Koszul quiver algebra can be used to determine if the Hochschild cohomology ring is finitely generated or not. 
The simplest kind of quiver algebras, where the ideal is generated by monomials, has been completely investigated 
in \cite{green2004hochschild}. It would therefore be of great interest to examine the next simplest case, 
in which the ideal may also include commutativity or anti-commutativity relations, especially since Xu's first 
counterexample belongs to this class. This is the purpose of the present paper, 
and we call such algebras partly (anti-)commutative. 

We describe the center of partly (anti-)commutative quiver algebras and give necessary
conditions for finite generation of the Hochschild cohomology ring when the algebra is a finite-dimensional
Koszul algebra.

We summarize our main results, for simplicity stated in the commutative case only. Similar results hold
in the anti-commutative case.

Let $Q$ be a finite quiver. A commutativity ideal $I\subseteq KQ$ is an ideal generated by quadratic monomials 
and commutativity relations $ab-ba$. The generator graph, $\Gamma_I$, of $I$ is the graph with 
the arrows of $Q$ as vertices and directed edges $a\to b$ corresponding to non-zero monomial relations
$ab\in I$ and undirected edges $a\und b$ corresponding to commutativity relations $ab-ba\in I$. 

\begin{thrm-adm}
A commutativity (or anti-commutativity) ideal, $I$, is admissible if 
and only if 
the generator graph, $\Gamma_{I^\perp}$, corresponding to the orthogonal ideal $I^\perp$, does 
not contain any directed cycle. 
\end{thrm-adm}

\begin{thrm-vsp}
 Let $I$ be a square-free commutativity ideal.  
The positively graded part of the center of $KQ/I$ has a basis given by all non-zero products $a_1a_2\dots a_k$, of
loops with the same basepoint, such that
\begin{enumerate}
\item All $a_i$ commute non-trivially modulo $I$: $a_ia_j=a_ja_i\neq 0$ for all $i$ and $j$. 
\item For all arrows $b$ in the quiver, one of the following two options holds: 
\begin{itemize}
\item $b$ commutes with all $a_i$. 
\item There exist $i$ and $j$ such that $a_ib=0=ba_j$.
\end{itemize}
\end{enumerate}
\end{thrm-vsp}

Let $Q_x$ be the subquiver
of $Q$ consisting of the point $x$, 
arrows $a$ such that $\mathfrak{o}(a)=x$ or $\mathfrak{t}(a)=x$ and the 
vertices where these arrows have their origin/target. Let $I_x=I\cap KQ_x$. 
A necessary contition for finite generation of $Z(KQ/I)$ is given in:
\begin{thrm-fg-cen}
 Suppose $I$ is a commutativity ideal such that $I^\perp$ is admissible. If $Z(KQ/I)$ is finitely
  generated as a $K$-algebra then for all $x$ in $Q_0$ either $Z(KQ_x/I_x)$ is trivial
  or there exists a non-empty set, $S$, of arrows $a$, such that
  $$I_x\supseteq \langle ab-ba, ca, ad
 \rangle_{\tiny{\begin{array}{l}a\in S,\\ b \textrm{ loop with basepoint }x,\\ c,d\textrm{ arrows such that }\ori(c)\neq\tar(c)=x, \tar(d)\neq\ori(d)=x
           \end{array}}}
.$$
\end{thrm-fg-cen}

In order to also give a sufficient condition, we introduce the relation graph $\Gamma_\mrm{rel}$ as
$\Gamma_I$ enhanced with directed edges $a\to b$ for those monomial $ab$ that are zero already in $KQ$.
\begin{thrm-rel}
  Let $I$ be a commutativity ideal such that $I^\perp$ is admissible. Consider the relation graph 
 $\Gamma_\mrm{rel}$. 
 \begin{itemize}
  \item[(i)] We have that $a_1a_2\dots a_n\in Z(KQ/I)$ if and only if $\{a_1, \dots, a_n\}$ is 
  a clique of loops in $\Gamma_\mrm{rel}$ and for any other vertex $b$
  in the graph either 
  \begin{itemize}
  \item there exists a directed edge from some $a_i$ in the clique to $b$ and a directed
  edge from $b$ to some $a_j$ in the clique \textbf{or}
  
  \item $\{a_1, \dots, a_n, b\}$ is also a clique in $\Gamma_\mrm{rel}$. 
 \end{itemize}
    \item[(ii)] $Z(KQ/I)$ is finitely generated if and only if whenever we have a clique fulfilling the
conditions in (i), each of its vertices fulfills these conditions separately, considered as one-element cliques.  
 \end{itemize}
\end{thrm-rel}

Finally, when $\La=KQ/I$ is a Koszul algebra the following theorem gives us the connection we need to use the
preceding theorems to determine if $\HH(\La)/\Ni$ is finitely generated or not.
\begin{thrm-hoch}
  Let $I$ be an (anti-)commutativity ideal and $\La=KQ/I$ a Koszul algebra. The 
 Hochschild cohomology ring of $\La$ modulo nilpotence, $\HH(\La)/\Ni$, is finitely generated if and only if
 $Z(\La^!)$ is finitely generated. 
\end{thrm-hoch}

\section{(Anti-)commutativity ideals and partly \mbox{(anti-)} commutative quiver algebras}

We always assume that $Q$ is a finite, connected quiver and $K$ is an algebraically closed field. 
If nothing else is written $a_i$ denotes an arrow. When we use the term monomial we always refer to
a path with coefficient $1$. In this paper, a binomial is a sum or difference of two monomials, i.e. a
sum or difference of two paths. 

\begin{defn}
A \textbf{commutativity ideal} is an ideal generated by quadratic monomials $a_ia_j$ and relations of the form $a_ka_l-a_la_k$. 
An \textbf{anti-commutativity
ideal} is an ideal generated by quadratic monomials $a_ia_j$ and relations of the form $a_ka_l+a_la_k$. 
\end{defn}
If $\mrm{char}K=2$ we have that every anti-commutativity ideal is also a commutativity ideal. When working in 
characteristic $2$ one may consider all such ideals to be commutativity ideals. 

\begin{defn}
 A \tbf{minimal generating set} of an (anti-)commutativity ideal is a set 
 $I_2=\{ a_ia_j,a_ka_l-a_la_k\}_{i,j,k,l}$ (or $I_2=\{ a_ia_j,a_ka_l+a_la_k\}_{i,j,k,l}$)
 such that $a_ka_l-a_la_k\in I_2$ (or $a_ka_l+a_la_k\in I_2$) implies
 that $a_ka_l\notin I_2$ and $a_la_k\notin I_2$. 
\end{defn}
The generators of the form $a_ia_j$ are called the monomial generators and the generators of the
form $a_ka_l-a_la_k$ and $a_ka_l+a_la_k$ are called the \mbox{(anti-)} commutativity generators. 

Any $x\in I$ can be written as 
$$x=\sum_{i,j}p_{ij}a_ia_jq_{ij}+\sum_{k,l}p_{kl}(a_ka_l-a_la_k)q_{kl}$$
where $p_{ij},q_{ij},p_{kl}$ and $q_{kl}$ are paths in $Q$.

\begin{defn}
  Assume $a_ia_j\notin I$. If $a_ia_j-a_ja_i\in I_2$ or $a_ia_j+a_ja_i\in I_2$
  we say that the transposition $(a_ia_j)$ is an \tbf{allowed transposition}. 
  Two paths, $p,q$, in $KQ$ are \tbf{equivalent}, denoted $p\sim q$, if $p$ can be obtained from $q$ by allowed 
  transpositions. 
 \end{defn}
It is easy to see that $\sim$ is an equivalence relation.

\begin{lemma}\label{comm-mono}
 Let $I$ be a commutativity ideal with minimal generating set $\langle a_ia_j,a_ka_l-a_la_k\rangle_{i,j,k,l}$. 
 Any monomial $m\in I$ is equivalent to 
 a monomial of the form $pa_ia_jq$, where $a_ia_j\in I$ and
 $p,q$ are paths. Conversely, all monomials of these types lie in $I$.
\end{lemma}
\begin{proof}
Assume $m\in I$. Then 
$$m=\sum_{i,j}p_{ij}a_ia_jq_{ij}+\sum_{k,l}p_{kl}(a_ka_l-a_la_k)q_{kl}$$
for some paths $p_{ij},q_{ij},p_{kl}$ and $q_{kl}$.
Since $m$ is a monomial either $m=pa_ia_jq$, where $p,q$ paths, or we have cancellations in the 
expression.
Assume that $m\neq pa_ia_jq$ for each $a_ia_j\in I$, then
$m=m_1-(m_1-m_2)-(m_2-m_3)-(m_3-m_4)-\dots (m_{i-1}-m_i)$,
where $m_i=m$ and $m_j-m_{j+1}=p(a_ka_l-a_la_k)q$ for some $p,q\in KQ$ and
$a_ka_l-a_la_k\in I_2$. 
By definition $m_j\sim m_{j+1}$ and
hence $m_1\sim m_2\sim m_3\sim \dots \sim m_i=m$. By assumption $m_1$ is of the form
$pa_ia_jq$, and hence $m$ is equivalent to a monomial of the form $pa_ia_jq$.

If $m\sim pa_ia_jq$ and $a_ia_j\in I$, then it is obvious that $m\in I$.

 \end{proof}
 \begin{lemma}\label{anti-comm-mono}
 Let $I$ be an anti-commutativity ideal with minimal generating set $\langle a_ia_j,a_ka_l+a_la_k\rangle_{i,j,k,l}$. 
 Any monomial $m\in I$ is equivalent to 
 a monomial of the form $\pm pa_ia_jq$, where $a_ia_j\in I$ and
 $p,q$ are paths. Conversely, all monomials of these types lie in $I$.
\end{lemma}
\begin{proof}
Analogous to the proof of Lemma \ref{comm-mono}.
 \end{proof}
 
 \begin{cor}\label{cor-square}
  If $a_1^2a_2a_3a_4\dots a_n\in I$ and $a_1a_2\dots a_n\notin I$ then $a_1^2\in I$.
 \end{cor}
\begin{proof}
 This follows immediately from \ref{comm-mono} and \ref{anti-comm-mono}.
\end{proof}
   
\begin{lemma}\label{comm-bin}
 Let $I$ be a commutativity ideal. Any binomial in $I$ is either the sum of two monomials
 in $I$ or of the form $b-c$, where $b$ and $c$ are monomials and $b\sim c$.

\end{lemma}
\begin{proof} 

Assume $b-c\in I$ with $b,c\notin I$. Then $b-c=\sum_{k,l}p_{kl}(a_ka_l-a_la_k)q_{kl}$, where
$a_ka_l-a_la_k$ are the binomial generators of $I$. We have that $b$ and $c$ are monomials, and hence 
$b-c=m_1-m_2+m_2-m_3+\dots m_{i-1}-m_i$, where $b=m_1$, $c=m_i$ and each 
$m_{j}-m_{j+1}=p(a_ka_l-a_la_k)q$ for some generator $(a_ka_l-a_la_k)$ and some paths $p,q$. 
By definition $m_j\sim m_{j+1}$ for all $j$, and hence $b\sim c$. 

 If $b\sim c$ then it is obvious that $b-c\in I$.  
 \end{proof}

 If $I$ is an anti-commutativity ideal we get a slight modification of the result in Lemma \ref{comm-bin}. 
 
 \begin{lemma}\label{anti-comm-sign}
  Let $I$ be an anti-commutativity ideal. Then any binomial in $I$ is either 
  \begin{itemize}
   \item[(i)] a sum of two monomials in $I$ or
   \item[(ii)] of the form $m+m'$ where $m\sim m'$ are monomials and the number of transpositions 
   used to transform $m$ to $m'$ is odd or
   \item[(iii)] of the form $m-m'$ where $m\sim m'$ are monomials and the number of transpositions 
   used to transform $m$ to $m'$ is even. 
  \end{itemize}
   \end{lemma}
 \begin{proof}
 Suppose $\sum_{k,l}p_{kl}(a_ka_l+a_la_k)q_{kl}$ is a binomial. Then 
 $\sum_{k,l}p_{kl}(a_ka_l+a_la_k)q_{kl}=m_1+m_2-(m_2+m_3)+m_3+\dots+(-1)^{n}(m_{n-1}+m_{n})$, 
 where $m_j\sim m_{j+1}$ for all $j$ and
 any pair $m_j+m_{j+1}$ correspond to exactly $1$ transposition. Hence 
 $\sum_{k,l}p_{kl}(a_ka_l+a_la_k)q_{kl}=m_1+(-1)^{n}m_{n}$, where $m_{n}$ is obtained from $m_1$
 by $n-1$ transpositions.
  \end{proof}

\begin{defn}
Let $I$ be an (anti-)commutativity ideal. The algebra $KQ/I$ is 
then said to be a \tbf{partly (anti-)commutative algebra}.
\end{defn}

\section{Admissible ideals}

If $I$ is an admissible ideal, then $KQ/I$ is a finite-dimensional algebra. In this section we use a 
combinatorial description of the (anti-)commutativity ideal to determine if it is admissible or not.
\begin{defn}
To every commutativity or anti-commutativity ideal $I$ we associate a directed graph $\Gamma_I$, called the 
\textbf{generator graph},
in the following way:
\begin{itemize}
\item To every arrow $a\in Q_1$ there is a vertex which we also call $a$.
\item To every monomial $ab\in I_2$ we associate a directed edge from $a$ to $b$, $a\to b$ 
(the monomial $a^2$ will give rise to a loop 
in the graph).
\item To every (anti-)commutativity relation $ab-ba\in I_2$ or $ab+ba\in I_2$ we associate an 
undirected edge between $a$ and $b$, $a\und b$. 
\end{itemize}
\end{defn}

\begin{exmp}
 \begin{itemize}
  \item[(i)] Let $I=\langle a^2, b^2, ab-ba,ac\rangle$. Then $\Gamma_I$ is
  
  $$
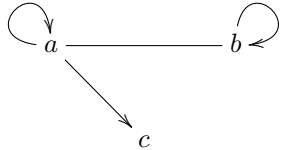
\xymatrix{
a  \ar@{-}[rr] \ar[dr] \ar@(l,u)[] & & b \ar@(u,r)[] \\ & c & }$$

  \item[(ii)] Let $I=\langle a^2, b^2, ab,ba,cd,da,bc\rangle$. The $\Gamma_I$ is
  
  $$
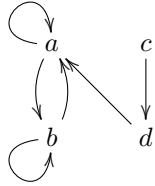
\xymatrix{
a   \ar@/_/[d] \ar@(l,u)[] & c  \ar[d] \\ b \ar@/_/[u] \ar@(l,d)[]  & d \ar[ul] \\&}$$
  
  \item[(iii)] Let $I=\langle a^2, b^2, c^2, d^2, ab+ba, ac+ca, ad+da, bc+cb, bd+dc \rangle$. Then $\Gamma_I$ is
  
  $$
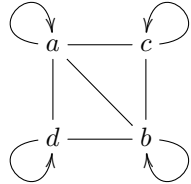
\xymatrix{a \ar@(l,u)[] \ar@{-}[r] \ar@{-}[d] \ar@{-}[dr]& c \ar@{-}[d] \ar@(r,u)[]\\ d \ar@{-}[r]\ar@(l,d)[]& b \ar@(r,d)[]}$$
  
 \end{itemize}~\\

\end{exmp}

If there exists a non-empty sequence of directed edges $a_1\to a_2\to \dots \to a_n\to a_1$ from $a_1$ to $a_1$
we say that the generator graph
contains a \tbf{directed cycle}. 

The generator graph that tells us about the admissibility of $I$ is the generator graph of the orthogonal 
ideal, $I^\perp$.

\begin{defn}
Let $I$ be an (anti-)commutativity ideal with minimal generating set $I_2$. The \tbf{orthogonal ideal} 
$I^\perp$ is defined by
\begin{itemize}
 \item If $ab-ba\in I_2$ then $ab+ba\in I^\perp_2$. 
 \item If $ab+ba\in I_2$, then $ab-ba\in I_2^\perp$.
  \item If $cd\neq 0$ and $cd\neq dc$ in $KQ/I$, 
  then $cd\in I_2^\perp$.
 \end{itemize}

\end{defn}

\begin{exmp}\label{ex1}
  \begin{itemize}
  \item[(i)] Let $Q$ be the following quiver
  $$
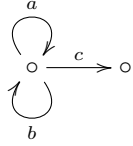
\xymatrix{\circ \ar@(dr,dl)[]^{b}\ar@(ul,ur)^{a}[] \ar[r]^{c} & \circ}$$
  and $I=\langle a^2, b^2, ab-ba,ac\rangle$. Then $I^\perp=\langle ab+ba, bc\rangle$. 
    
  \item[(ii)] Let $Q$ be the following quiver
  $$
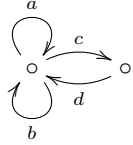
\xymatrix{\circ \ar@(dr,dl)[]^{b}\ar@(ul,ur)^{a}[] \ar@/^/[r]^{c} & \circ \ar@/^/[l]^{d}} $$
  and
  $I=\langle a^2, b^2, ab,ba,cd,da,bc\rangle$. Then $I^\perp=\langle ac, db, dc\rangle$.
  
  \item[(iii)] Let $Q$ be the following quiver
  $$
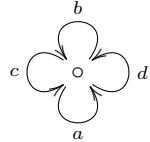
\xymatrix{\circ \ar@(dl,dr)[]_{a} \ar@(ur,ul)[]_{b} 
  \ar@(dr,ur)[]_{d} 
  \ar@(ul,dl)[]_{c}
  }$$
  and
  $I=\langle a^2, b^2, c^2, d^2, ab+ba, ac+ca, ad+da, bc+cb, bd+db \rangle$. Then
  $I^\perp=\langle cd,dc, ab-ba, ac-ca, ad-da, bc-cb, bd-db \rangle$.
  
 \end{itemize}~\\
\end{exmp}


\begin{prop}\label{cup-graph}
Any non-zero path $ab$ of length $2$ in $KQ$ is represented in the generator graphs $\Gamma_I$ 
or $\Gamma_{I^\perp}$ by exactly one of the following:
\begin{itemize}
 \item[(i)] Two undirected edges $a\und b$, one in $\Gamma_I$ and one in $\Gamma_{I^\perp}$, or
 \item[(ii)] a directed edge $a\to b$ in $\Gamma_I$, or
 \item[(iii)] a directed edge $a\to b$ in $\Gamma_{I^\perp}$.
\end{itemize}
\end{prop}
\begin{proof}
 If $ab$ is a non-zero path, then, since $ab$ is a basis element in $KQ_2$, by 
 the definition of $I^\perp$ 
 we have that either $ab\in I$ or $ab\in I^\perp$ or there exist \mbox{(anti-)}
 commutativity
 relations in the ideals $I$ and $I^\perp$. The definition of $\Gamma_I$ tells
 us how to represent these three cases and we get the result in the proposition.
\end{proof}

\begin{lemma}\label{square-free-admissible}
 If $I$ is an admissible commutativity ideal, then $I$ contains all non-zero squares in $KQ$ and
 $I^\perp$ is a square-free anti-commutativity ideal. If $I$ is an admissible anti-commutativity ideal,
 then $I$ contains all non-zero squares in $KQ$ and
 $I^\perp$ is a square-free commutativity ideal. 
\end{lemma}
\begin{proof}
 Assume that $I$ is admissible, but does not contain $a^2\neq 0$. Then, by Lemma \ref{comm-mono}, 
 it will not contain $a^n$ for any $n$ and hence we get a contradiction to the assumption that $I$ was
 admissible. This means that an admissible ideal $I$ always has to contain all non-zero squares and 
 we can conclude that $I^\perp$ will be square-free.
\end{proof}

\begin{lemma}\label{simple-corollary}
Let $I$ be an (anti-)commutativity ideal. 
 \begin{itemize}
  \item[(i)] If $a_1a_2\dots a_n\in I$ we have that $a_ia_{i+1}\notin I^\perp$ for some $i$.
  \item[(ii)] If $a_1a_2\dots a_n\notin I$, then for any $1\le i\le n-1$ there exists an edge 
  (directed or undirected) from $a_i$ to $a_{i+1}$ in $\Gamma_{I^\perp}$.
 \end{itemize}
\end{lemma}
\begin{proof}
 \begin{itemize}
  \item[(i)] By Lemma \ref{comm-mono} we have that $a_1a_2\dots a_n\sim pa_ka_lq$ for some $a_ka_l\in I$ and
  some paths $p$ and $q$. Either $a_1a_2\dots a_n=pa_ka_lq$ and then there exists a pair $a_ia_{i+1}=a_ka_l\in I$ 
  which implies $a_ia_{i+1}\notin I^\perp$, or there exists an allowed transposition $(a_ia_{i+1})$. An allowed 
  transposition corresponds to a commutativity relation in $I$, and $a_ia_{i+1}-a_{i+1}a_i\in I$ gives 
  $a_ia_{i+1}+a_{i+1}a_i\in I^\perp$, which implies that $a_ia_{i+1}\notin I^\perp$.
  
  \item[(ii)] If $a_1a_2\dots a_n\notin I$ we have that $a_ia_{i+1}\notin I$ for all $1\le i\le n-1$. 
  By Proposition \ref{cup-graph} this implies that we either have a directed edge $a_i\to a_{i+1}$ or 
  an undirected edge $a_i\und a_{i+1}$ in 
  $\Gamma_{I^\perp}$.
 \end{itemize}

\end{proof}

\begin{lemma}\label{loops-rel}
 If $a_i\und a_j$ is an undirected edge in $\Gamma_I$, then $a_i$ and $a_j$ are loops
 at the same vertex. 
 For any pair of loops $a_i$ and $a_j$ at the same vertex both $a_ia_j$ and $a_ja_i$
 have to be represented by a directed edge in either 
$\Gamma_I$ or $\Gamma_{I^\perp}$ or undirected edges in both. 
\end{lemma}
\begin{proof}
If $a_i\und a_j$ is an undirected edge in $\Gamma_I$ we have that $a_ia_j-a_ja_i\in I_2$, i.e.
$a_ia_j-a_ja_i$ is a relation which means that $\mathfrak{o}(a_i) = \mathfrak{o}(a_j)$
and $\mathfrak{t}(a_i) = \mathfrak{t}(a_j)$. 
Since $a_ia_j\neq 0$ and $a_ja_i\neq 0$ we have $\tar(a_i)=\ori(a_j)$ and $\tar(a_j)=\ori(a_i)$.
Hence $a_i$ and $a_j$ are loops
 at the same vertex. 

If $a_i$ and $a_j$ are loops at the same vertex, then $a_ia_j$ and $a_ja_i$ are never trivially $0$
and hence, by Proposition \ref{cup-graph}, we have that $a_ia_j$ is represented by a directed edge in either 
$\Gamma_I$ or $\Gamma_{I^\perp}$ or undirected edges in both. 
\end{proof}

\begin{lemma}\label{help-adm}
 Let $I$ be an admissible (anti-)commutativity ideal. 
 Assume that the path $a_1a_2\dots a_na_1$ is not contained in $I$ such that $a_i\neq a_j$ for $i\neq j$. 
 Then there exists a path $a_1a_ia_{i+1}\dots a_ka_1\notin I$ such that $a_ia_{i+1}\dots a_k$ 
 is a subpath of $a_1a_2\dots a_na_1$, $a_1a_i\in I^\perp$ and $a_ka_1\in I^\perp$.
\end{lemma}
\begin{proof}
Assume $a_1a_j\notin I^\perp$ for all $1\le j\le n$. Let $p_n=a_1a_2\dots a_na_1$.
Then, by Lemma \ref{simple-corollary}(ii), we would have
an allowed transposition $(a_1a_2)$ and hence $p_n$ is equivalent to $a_2a_1a_3a_4\dots a_{n}a_1$. By 
Lemma \ref{simple-corollary}(ii) we then have an edge $a_1a_3$ and since we assume that there 
is no directed edge from $a_1$ we get that $(a_1a_3)$ is an allowed transposition. Inductively
we get that $(a_1a_i)$ are allowed transpositions
for all $2\le i\le n$ and hence
$p_{n}\sim  a_2a_3\dots a_{n}a_1^2$.
By Lemma \ref{square-free-admissible} we have that $a_1^2\in I$ and hence $p_n\in I$ and we have a 
contradiction. Hence 
there exist an $a_i$ such that $a_1a_i\in I^\perp$. Assume that $i$ is as small as possible, i.e. 
$a_1a_m\notin I^\perp$ for $m<i$. Then 
$$a_1a_2\dots a_k\sim a_2a_3\dots a_{i-1}a_1a_ia_{i+1}\dots a_{n}a_1,$$ which 
implies that $a_1a_ia_{i+1}\dots a_{n}a_1\notin I$. 

Now assume that $a_ja_1\notin I^\perp$ for every $1\le j\le n$.
By the same inductive argument as above, get that $(a_ja_1)$ are allowed transpositions
all $1<j<n$.
This gives a contradiction since $a_1a_i\in I^\perp$ and  this implies that 
$a_1a_i-a_ia_1\notin I_2$. Hence
there exists an $a_k$ such that $a_ka_1\in I^\perp$. By assumption above $(a_1a_m)$ are
allowed transpositions for $m<i$, hence $k\ge i$.
Assume that $k$ is as big as possible, i.e.
$a_ja_1\notin I^\perp$ for $j>k$. Then $a_1a_ia_{i+1}\dots a_{n}a_1\sim a_1a_ia_{i+1}\dots 
a_ka_1a_{k+1}a_{k+2}\dots a_{n}$ and hence we have that $a_1a_ia_{i+1}\dots a_ka_1\notin I$.
\end{proof}

\begin{prop}\label{shortcut}
 Let $I$ be an (anti-)commutativity ideal and $a_1a_2\dots a_n\notin I$. 
 If $$a_1a_2\dots a_n\sim a_1a_2\dots a_{i-1}a_{i+1}a_ia_{i+2}a_{i+3}\dots a_n$$
 then there exist edges (directed or undirected) from $a_{i-1}$ to $a_{i+1}$ and from $a_i$ to $a_{i+2}$
 in $\Gamma_{I^\perp}$. 
\end{prop}
\begin{proof}
 Assume
 $$a_1a_2\dots a_n\sim a_1a_2\dots a_{i-1}a_{i+1}a_ia_{i+2}a_{i+3}\dots a_n.$$
 Then we have an allowed transposition $(a_ia_{i+1})$ which implies that $a_ia_{i+1}-a_{i+1}a_i\in I$ and hence
 $a_i$ and $a_{i+1}$ are loops. By Lemma
\ref{loops-rel} we have that there exist edges from $a_{i-1}$ to $a_{i+1}$ and from $a_i$ to $a_{i+2}$
in either $\Gamma_{I^\perp}$ or $\Gamma_I$. Since we assumed that $a_1a_2\dots a_n\notin I$
we get that the edges have to lie in $\Gamma_{I^\perp}$.
\end{proof}
It might be good to visualize what this proposition acctually tells us. Assume that 
$a_1a_2\dots a_n\notin I$. Then, by Lemma \ref{simple-corollary}(ii) we have that
there exist edges $a_ia_{i+1}$ in the generator graph $\Gamma_{I^\perp}$. They can be either directed
or undirected depending on the generators of $I$, but for example it might look like this:
$$\xymatrix{ a_1\ar[r]&a_2\ar@{-}[r] & a_3\ar@{-}[r] & a_4\ar[r] &\dots \ar[r]& a_n.    
}$$
Now, what Proposition \ref{shortcut} says is that whenever we have an undirected edge $a_i\und a_{i+1}$,
we also have edges from $a_{i-1}$ to $a_{i+1}$ and from $a_i$ to $a_{i+2}$. Since the edge $a_2\und a_3$ 
is undirected in
the graph above we get two more edges:
$$\xymatrix{a_{1} \ar@{--}@/_1pc/[rr] \ar[r]& a_2 \ar@{-}[r] \ar@{--}@/^1pc/[rr] 
& a_{3}\ar@{-}[r] & a_4 \ar[r]&\dots\ar[r] & a_n. }$$
These edges can be either directed or undirected, depending on the ideal $I$. For example
if $a_1\ \to a_3$ is directed and $a_2\und a_4$ is undirected we get the following picture:
$$\xymatrix{a_{1} \ar@/_1pc/[rr] \ar[r]& a_2 \ar@{-}[r] \ar@{-}@/^1pc/[rr] 
& a_{3}\ar@{-}[r] & a_4 \ar[r]&\dots\ar[r] & a_n. }$$
We can use these new edges to construct more paths that are not contained in $I$, as seen in the 
following corollary.

\begin{cor}\label{shorter-path}
 Let $I$ be an (anti-)commutativity ideal. Let $a_1a_2\dots a_n\notin I$ be a path and let $a_i\und a_{i+1}$ be an undirected
 edge in the generator graph $\Gamma_{I^\perp}$. Then $$a_1a_2\dots a_{i-1}a_{i+1}a_{i+2}\dots a_n\notin I.$$ 
\end{cor}
\begin{proof}
 If there exists an edge from $a_j$ to $a_k$ in $\Gamma_{I^\perp}$ then we have that $a_ja_k\notin I$. 
 By Lemma \ref{simple-corollary}(ii) we have edges from $a_j$ to $a_{j+1}$ for any $1\le j\le n-1$ in
 $\Gamma_{I^\perp}$ and
 by Proposition \ref{shortcut} we have that there exists an edge from $a_{i-1}$ to $a_{i+1}$ in $\Gamma_{I^\perp}$.
 Hence $a_1a_2\dots a_{i-1}a_{i+1}a_{i+2}\dots a_n\notin I$. 
\end{proof}

\begin{defn}
 If $a_1a_2\dots a_n\sim a_1a_2\dots a_{i-1}a_{i+1}a_ia_{i+2}a_{i+3}\dots a_n$
 we say that there exist a \tbf{shortcut} between $a_{i-1}$ and $a_{i+1}$ in the generator graph $\Gamma_{I^\perp}$. 
\end{defn}

The following theorem now gives us a way to determine if a commutativity (or anti-commutativity) ideal $I$ is admissible
by just a quick look at the generator graph for the orthogonal ideal $I^\perp$.   
\begin{thm}\label{adm}
A commutativity (or anti-commutativity) ideal, $I$, is admissible if 
and only if 
the generator graph, $\Gamma_{I^\perp}$, corresponding to the orthogonal ideal $I^\perp$, does 
not contain any directed cycle. 
\end{thm}
\begin{proof}
An ideal fails to be admissible if for any $n$ there is at least one path, $p_n$, of length $n$ such that 
$p_n\notin I$. 
It follows from Lemma \ref{cup-graph} and the definition of $I^\perp$ 
that directed edges in $\Gamma_{I^\perp}$ correspond to paths of 
length $2$ not
contained in $I$. 
A path, $a_1a_2\dots a_n$ built from generators $a_ia_{i+1}\in I^\perp$ is clearly not contained in $I$ 
(by Lemma \ref{simple-corollary}(i)).
If $\Gamma_{I^\perp}$ contains a directed cycle we hence 
can construct paths of arbitrary length that are not contained in $I$. 

Assume that 
for any $n\ge 1$ we are able to find a 
path, $p_n\notin I$, of length $n$. Suppose $n> |Q_1|$, then there exists an arrow $a_1$ such that $a_1$
repeats in $p_n$. Let $a_1$ be the first arrow that repeats and let $p_{k}$ be the part 
of $p_n$ that starts with the first copy of $a_1$ and ends with the second copy $a_{k}=a_1$.
By Lemma \ref{help-adm} there exists a path 
$p=a_1a_2\dots a_{k-1}a_1\notin I$ such that $a_2a_3\dots a_{k-1}$ is a subpath of $p_n$, 
$a_1a_2\in I^\perp$ and $a_{k-1}a_1\in I^\perp$.

We will prove the existence of a directed cycle in $\Gamma_{I^\perp}$ by showing that there exist a directed shortcut
past any undirected edge in $\Gamma_{I^\perp}$. This will be done by induction over the length of $p$. 
If $\mathfrak{l}(p)=4$, then $p=a_1a_2a_3a_1$ with $a_1a_2\in I^\perp$ and $a_3a_1\in I^\perp$. By 
Lemma \ref{simple-corollary} (ii) we have an edge between $a_2$ and $a_3$ in $\Gamma_{I^\perp}$. 
If $a_2a_3\in I^\perp$ we
have no undirected edges in the path $a_1a_2a_3a_1$ and hence we have a directed cycle in $\Gamma_{I^\perp}$. 
Assume that $a_2\und a_3$ is an undirected edge. Then
$p\sim a_1a_3a_2a_1$ and by Lemma \ref{simple-corollary} (ii) we have that there exist an edge from $a_1$ to $a_3$ in
$\Gamma_{I^\perp}$. By Proposition \ref{cup-graph} this edge can't be undirected and hence we have a directed shortcut 
past the undirected edge $a_2\und a_3$ that gives a directed cycle in $\Gamma_{I^\perp}$. 

Assume that if $\mathfrak{l}(p)<k$ then we can find a directed shortcut past any undirected edge. Note 
that a sequence of directed shortcuts in a path implies a directed cycle 
in $\Gamma_{I^\perp}$.

Assume $p=a_1a_2\dots a_{k}$, where $a_k=a_1$, $a_1a_2\in I^\perp$ and $a_{k-1}a_1\in I^\perp$. Assume that
$a_i\und a_{i+1}$ is a pair of consecutive arrows in $p$ that correspond to an undirected edge in $\Gamma_{I^\perp}$. 
By Corollary \ref{shorter-path} we then have that $$a_1a_2\dots a_{i-1}a_{i+1}a_{i+2}\dots a_{k-1}a_1\notin I.$$
This path is a shorter path where we by our assumption can find directed shortcuts past any undirected edge.
Hence $\Gamma_{I^\perp}$ contains a directed cycle. \end{proof}

\begin{exmp}
 \begin{itemize}
  \item[(i)] Consider $I^\perp$ from Example \ref{ex1}(i), then $\Gamma_{I^\perp}$ is
    $$\xymatrix{
a  \ar@{-}[rr]   & & b \ar[dl] \\ & c & }$$
and hence $I$ is admissible.

  \item[(ii)] Consider $I^\perp$ from Example \ref{ex1}(ii), then $\Gamma_{I^\perp}$ is
    $$
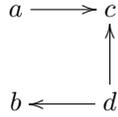
\xymatrix{
a   \ar[r] & c   \\ b  & d \ar[u] \ar[l] \\&}$$
and hence $I$ is admissible.
  
  \item[(iii)] Consider $I^\perp$ from Example \ref{ex1}(ii), then $\Gamma_{I^\perp}$ is
  $$
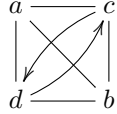
\xymatrix{a \ar@{-}[r] \ar@{-}[d]  \ar@{-}[dr]& c \ar@/_/[dl] \ar@{-}[d] \\ d \ar@/_/[ur]\ar@{-}[r]& b }$$
  This graph has a directed cycle $c\to d\to c$, and hence $I$ is not admissible. 
 \end{itemize}~\\

\end{exmp}

\section{Centers of partly (anti-)commutative quiver algebras}

The center of a graded algebra is also a graded algebra. 
If $|Q_0|>1$ and the quiver is connected
we have that no vertex lies in the
center, since for any arrow $a$ such that $\ori(a)\neq \tar(a)=x$ we have that $xa=0$ and $ax=a$.
We have that $Z^0(KQ/I)$ is generated as a $K$-algebra by the identity element and hence $Z^0(KQ/I)=K$ 
for all partly (anti-)commutative quiver algebras. 

\begin{defn}
 Let the positively graded part of the center be denoted by $Z^+(KQ/I)$, i.e. $$Z^+(KQ/I)=\bigoplus_{i\ge1}Z^i(KQ/I).$$ 
\end{defn}

We have that $Z(KQ/I)=Z^0(KQ/I)\bigoplus Z^+(KQ/I)$, and $Z^0(KQ/I)=K$ for all connected quiver algebras $KQ/I$, 
hence the rest of this section is devoted to describing $Z^+(KQ/I)$.   

We have seen several ways to decompose $KQ/I$ as a vector space, and in this chapter we need one more.
\begin{defn}
 Let $p$ be a non-zero monomial in $KQ$ of length at least $1$. Let $V_p$ be the vector subspace generated by the monomials consisting of
 permutations of the arrows of $p$ (many of them can be $0$ for trivial reasons or because they lie in $I$). 
 A linear combination of elements belonging to the same $V_p$ is said to be 
 \tbf{permutation homogeneous}. 
 An ideal generated by permutation 
 homogeneous elements is a \tbf{permutation homogeneous ideal}. 
\end{defn}
If we pick representatives $p$ for each permutation homogeneous set we can easily see that 
each monomial of positive grade lies in a specific $V_p$, i.e. $KQ=K\oplus \bigoplus_pV_p$.
Note that both commutativity ideals and anti-commutativity ideals are permutation homogeneous ideals.
\begin{lemma}\label{vsp-permhomo}
 Let $I$ be a permutation homogeneous ideal. Then 
 $$KQ/I\cong K\oplus \bigoplus_p V_p/(I\cap V_p)$$
 as a sum of vector spaces.
\end{lemma}
\begin{proof}
 If $I$ is permutation homogeneous it is clear that $I=\bigoplus_p I\cap V_p$ as a sum of vector spaces. 
 Hence $KQ/I\cong K\oplus \bigoplus_p V_p/(I\cap V_p)$.  
\end{proof}

\begin{lemma}\label{sums}
Let $I$ be a permutation homogeneous ideal. The positively graded part of the center, $Z^+(KQ/I)$, of $KQ/I$  is spanned by 
permutation homogeneous elements.
\end{lemma}
\begin{proof}
Assume $\sigma\in Z^+(KQ/I)$. Then $\sigma=\rho_1+\rho_2+\dots \rho_n$ where each 
$\rho_i\in V_{p_i}/(I\cap V_{p_i})$. 
Since $\sigma\in Z^+(KQ/I)$ for any $a$
we have $$a\rho_1+a\rho_2+\dots+a\rho_n=\rho_1a+\rho_2a+\dots+\rho_n a.$$
Since the sum is direct, we have that $a\rho_k=\rho_k a$ for all $1\le k\le n$.
Hence the center has to be a
permutation homogeneous ideal.\end{proof}

\begin{lemma}\label{cycles}
Let $I$ be a (anti-)commutativity ideal.
The positively graded part of center of $KQ/I$ is spanned by linear combinations of 
cycles. 
\end{lemma}
\begin{proof}
Assume $\sum_i\alpha_ip_i\in Z^+(KQ/I)$ and that all $p_i$ are linearly independent in the vector space $KQ/I$. 
Let $p_1=a_1a_2\dots a_n$. We shall now show that $p_1$ is 
a cycle. 
Assume that for any pair $p_i,p_j$ in 
the sum, we have that $p_i\neq p_j$, hence we have no cancellations
in the sum. 
We have that
$$\ori(p_1)\sum_i\alpha_ip_i=\sum_{\ori(p_i)=\ori(p_1)}\alpha_ip_i\neq 0$$
since, at least, $\ori(p_1)p_1=p_1\neq 0$. 
Hence $$\sum_i\alpha_ip_i\ori(p_1)=\sum_{\tar(p_i)=\ori(p_1)}\alpha_ip_i\neq 0$$ 
and hence $p_1\ori(p_1)$ is non-zero, i.e. $\ori(p_1)=\tar(p_1)$.
\end{proof}

\begin{lemma}\label{mono-center}
 Let $I$ be a square-free (anti-)commutativity ideal, i.e. an \mbox{(anti-)} commutativity ideal that doesn't contain any non-zero 
 monomials
 of the form $a_i^2$. 
 Then the positively graded part of the center of $KQ/I$ is spanned by non-zero monomials $a_1a_2\dots a_n$ of loops at the
 same basepoint such that
 $(a_ia_j)$ is an allowed transposition for any pair $a_i,a_j$ in the monomial.  
\end{lemma}
\begin{proof}
%
%
Let $g = \sum_i \alpha_i p_i\in Z^+(KQ/I)$ be a non-zero permutation homogeneous element, 
where $\alpha_i\in K$ and the $p_i$ are paths. Let
$a$ be any arrow occurring first in some $p_i$ and 
let $b$ be any other arrow. Let $p_j$ be such that 
there is a maximal number $m$ of copies of $a$ to the left of the first occurrence of $b$. 

Since $ag = ga$ and by Corollary \ref{cor-square} we have that $ap_j \neq 0$ 
it must be that $ap_j = p_k a$ for some $p_k$. 

There are $m+1$ copies of $a$ before
the first occurrence of $b$ in $ap_j$, and therefore also in $p_k a$. 
But, by maximality, $p_k$ contains at most $m$ copies of $a$ before the first $b$, 
and if $a$ and $b$ do not commute, the same must be true for $p_k a$. This gives 
a contradiction. 

Hence $a$ and $b$ commute, so that $a$ commutes with all arrows in the paths $p_i$.
A repetition of this argument shows that all arrows commute, and $g$ may be written 
as a single term $\alpha p$. Since all arrows of $p$ commute, it follows 
that they must be loops at the same basepoint, for otherwise the product would be $0$.
\end{proof}

\begin{thm}\label{vsp}
Let $I$ be a square-free commutativity ideal.  
The positively graded part of the center of $KQ/I$ has a basis given by all non-zero products $a_1a_2\dots a_k$, of
loops with the same basepoint, such that
\begin{enumerate}
\item All $a_i$ commute non-trivially modulo $I$: $a_ia_j=a_ja_i\neq 0$ for all $i$ and $j$. 
\item For all arrows $b$ in the quiver, one of the following two options holds: 
\begin{itemize}
\item $b$ commutes with all $a_i$. 
\item There exist $i$ and $j$ such that $a_ib=0=ba_j$.
\end{itemize}
\end{enumerate}
\end{thm}
\begin{proof}
 We begin with proving that all such elements lie in the center. Let $p=a_1a_2\dots a_n$
 be a monomial of the type described in the theorem and let $b$ be an arrow such that $b$
 does not commute with all $a_l$. 
 Since all the arrows in $p$ commute $a_1a_2\dots a_n\sim 
 a_ja_1a_2\dots a_{j-1}a_{j+1}\dots a_n$. Hence $ba_j=0$ implies $bp=0$.
 In the same way $a_ib=0$ implies $pb=0$. 
 
 By Lemma \ref{mono-center} we have that the center is spanned by monomials of loops where 
 all arrows commute non-trivially. We now show that these monomials also have property $2$ above.

 Let $p=a_1a_2\dots a_n\in Z^+(KQ/I)$ be a monomial such that $a_ia_j=a_ja_i\neq 0$ for all $1\le i,j\le n$. 
 Assume $b\in Q_1\setminus\{a_1, a_2, \dots, a_n\}$ and $bp=pb$. 
We have four cases:
\begin{enumerate}
 \item Assume $\mathfrak{o}(b)=\mathfrak{t}(b)=\mathfrak{o}(p)$. Then either  
$bp\sim pb$, which implies that $b$ commutes with all $a_i$ in the monomial, or $bp\in I$ and $pb\in I$.
Since $I$ is a quadratic ideal, $b\notin I$ and by assumption $p\notin I$. 
Hence if $bp=0$ by Lemma \ref{comm-mono} we get that there exists an $a_i$ in the monomial such that 
$ba_i=0$. In 
the same way we get that there exists an $a_j$ in the monomial such that $a_jb=0$. 
\item Assume $\mathfrak{o}(b)= \mathfrak{o}(p)$ and $\mathfrak{t}(b)\neq \mathfrak{o}(p)$. Then
$bp=0$ trivially and hence $ba_i=0$ for all $a_i$ in the path. As above there exists an $a_j$ in the monomial 
such that $a_jb=0$.
\item Assume $\mathfrak{o}(b)\neq \mathfrak{o}(p)$ and $\mathfrak{t}(b)= \mathfrak{o}(p)$. Analogous to 
case 2 we get that there exists an $a_i$ in the monomial such that $ba_i=0$ and $a_jb=0$ for all $a_j$
in the monomial.
\item Assume $\mathfrak{o}(b)\neq \mathfrak{o}(p)$ and $\mathfrak{t}(b)\neq \mathfrak{o}(p)$. Then
$ba_i=a_ib=0$ for all $a_i$ in the monomial.
\end{enumerate}
That the set of monomials fulfilling these conditions is a basis is seen by realizing that 
if $p,q\in V_p$ and $p,q\in Z^+(KQ/I)$, then $p=q$ and hence 
there is at most one monomial from each permutation homogeneous subspace $V_p$ in $Z^+(KQ/I)$
and since $KQ/I\cong K\oplus \bigoplus_p V_p/(I\cap V_p)$ we get that these form a basis for $Z^+(KQ/I)$.
\end{proof}

To prove the anti-commutative version of Theorem \ref{vsp} we need the following lemma. 
\begin{lemma}\label{anti-comm-path}
 Let $I$ be a square-free anti-commutativity ideal. Assume that $a_1a_2\dots a_n\neq 0$. 
 If $$a_ia_1a_2\dots a_n=a_1a_2\dots a_na_i$$
 for every $a_i$ in the monomial, then if $n$ is even, there are an even number of copies 
 of every $a_i$ in the monomial and if $n$ is odd, there are an odd number of copies
 of every $a_i$ in the monomial. 
 \end{lemma}
\begin{proof}
If $ba_1a_2\dots a_n=a_1a_2\dots a_nb$, then by Lemma \ref{anti-comm-sign} we need an even number
of transpositions to obtain $a_1a_2\dots a_nb$ from $ba_1a_2\dots a_n$. 
If $n$ is even we need an even number of
copies (or no copy) of $b$ in the monomial $a_1a_2\dots a_n$, since an even number of copies
gives an even number of transpositions. If $n$ is odd the number of copies of $b$
must be odd in order to get an even number of transpositions. 
\end{proof}

\begin{thm}\label{vspsemi}
Let $I$ be a square-free anti-commutativity ideal. 
The positively graded part of the center of $KQ/I$ has a basis given by all non-zero products $a_1a_2\dots a_k$,
of loops with the same basepoint, such that
\begin{enumerate}
\item If $k$ is even:
\begin{itemize}
 \item The monomial contains an even number of each arrow, $a_i$.
 \item For all arrows $a_i, a_j$ in the monomial $a_ia_j=-a_ja_i\neq 0$. 
\item For all arrows $b$ in the quiver, one of the following two options holds: 
\begin{itemize}
\item $b$ anti-commutes non-trivially with all $a_i$, i.e. $ba_i=-a_ib\neq 0$. 
\item There exist $i$ and $j$ such that $a_ib=0=ba_j$.
\end{itemize}\end{itemize}
\item If $k$ is odd:
\begin{itemize}
\item The monomial contains an odd number of each arrow, $a_i$.
\item For all arrows $a_i,a_j$ in the monomial $a_ia_j=-a_ja_i\neq 0$.
\item For all other arrows $b$, there exist $i$ and $j$ such that $a_ib=0=ba_j$.
\end{itemize}
\end{enumerate}
\end{thm}
\begin{proof}
Assume that $p=a_1a_2\dots a_{2n}$ fulfills the conditions listed in item 1 above.
Then $bp=pb$ for all $b\in Q_1$ and hence it follows that $p\in Z^+(KQ/I)$.

By Lemma \ref{mono-center}, $Z^+(KQ/I)$ is spanned by non-zero monomials $a_1a_2\dots a_k$ of 
loops at the same base point such that all $(a_ia_j)$ are allowed transpositions. Consider
$p=a_1a_2\dots a_{2n}\in Z^+(KQ/I)$ of even length. Then $a_ia_1a_2\dots a_{2n}=a_1a_2\dots a_{2n}a_i$
for all $a_i\in \{a_1, a_2, \dots a_{2n}\}$ and by Lemma \ref{anti-comm-path} we have that 
$p$ contains an even number of copies of
every $a_i$.
Assume $b\notin \{a_1, a_2, \dots a_{2n}\}$. Since $p\in Z^+(KQ/I)$ we have that $bp=pb$.
Assume $bp\neq 0$, then $$ba_1a_2\dots a_{2n}=a_1a_2\dots a_{2n}b$$ and 
hence $(a_ib)$ has to be an allowed transposition for any $a_i$ in the monomial $p$ and
since the length of the monomial $p$ is even $bp=pb$. 
Assume $bp=0$, then, by analogous arguments as in the proof of Theorem \ref{vsp} we
get that there exist $a_i$ and $a_j$ such that $a_ib=ba_j=0$.

Now consider $p=a_1a_2\dots a_{n}\in Z^+(KQ/I)$ of odd length. 
By Lemma \ref{anti-comm-path} we have that $p$ contains an odd number of copies of
every $a_i$. Assume $b\notin \{a_1, a_2, \dots, a_n\}$. Then 
$bp=pb$ implies that $bp=pb=0$, since the number of transpositions needed to
rewrite $ba_1a_2\dots a_n$ to $a_1a_2\dots a_nb$ is odd. 
By analogous arguments as in the proof of Theorem \ref{vsp} we
get that there exist $a_i$ and $a_j$ such that $a_ib=ba_j=0$.

Let $p=a_1a_2\dots a_{n}\in Z^+(KQ/I)$ be a monomial of odd length such that it fulfills condition $2$ above.
Then for any $a\in KQ/I$ we get that $ap=pa$, i.e. $p\in Z(KQ/I)$.  

That the specified monomials consistute a basis follows from the fact that 
there is at most one monomial from each permutation homogeneous subspace $V_p$ in $Z^+(KQ/I)$
and since $KQ/I\cong K\oplus \bigoplus_p V_p/(I\cap V_p)$ we get that these form a basis for $Z^+(KQ/I)$.
\end{proof}

\begin{exmp}\label{cen-ex}
 Let $Q$ be the quiver in Example \ref{ex1}(i) and let $I=\langle ab+ba, bc\rangle$. Then $Z(KQ/I)$ has a basis
 consisting of $b^k$ for all $k$ and $a^{2i}b^{2j}$ for $i\ge0$ and $i\ge1$. 
\end{exmp}

In section \ref{the_graded_center} we need to have control over the nilpotent elements of $Z(KQ/I)$. The 
following proposition makes this easy.

\begin{prop}\label{nilpotent}
 Let $I$ be a square-free commutativity ideal. 
 Then $Z(KQ/I)$ does not contain
 any nilpotent elements.
\end{prop}
\begin{proof}

 Assume $\sum_{i=1}^n \alpha_ip_i\in Z(KQ/I)$, where $\alpha_i\in K$, each $p_i$ is a monomial
 of loops at the same basepoint and one of the $p_i$s can be of degree $0$. 
 By Theorem \ref{vsp} we have that $p_i\in Z(KQ/I)$ for all $1\le i\le n$.
 Assume $p_i\neq p_j$ for all $i\neq j$, i.e. 
 for any permutation homogeneous subspace $V_p$ there is at most one $p_i$ in the sum such that $p_i\in V_p$.
 By Theorem \ref{vsp} we have that $p^m=a_1^ma_2^m\dots a_n^m$ and since $I$ is square-free 
 we have that $p_i^m\neq 0$ for all $1\le i\le n$ and all $m\ge 1$.
 
 Order the monomials in $KQ/I$ according to the lexicographical order. 
 Let $p_1$ be the leading term in $\sum_{i=1}^n \alpha_ip_i\in Z(KQ/I)$. Since $p_1^m\neq 0$
 then $p_1^m$ will be the leading term in $\Big(\sum_{i=1}^n \alpha_ip_i\Big)^m$. 
 We have that $p_1^m\in V_{p_1^m}$, and since $p_1^m>_\mrm{lex}p_{i_1}p_{i_2}\dots p_{i_m}$ for all other terms in
 $\Big(\sum_{i=1}^n \alpha_ip_i\Big)^m$, we see that $p_1^m$ is the only path that lies in $V_{p_1^m}$. 
 This means that no other term in the product can cancel $p_1^m$ and hence
 $$\Big(\sum_{i=1}^n \alpha_ip_i\Big)^m=\alpha_1^mp_1^m+\textrm{other terms }\neq 0\qedhere$$
\end{proof}

By an analogous proof we get the same result when $I$ is a square-free anti-commutativity ideal.

\begin{prop}\label{nilpotent-anti}
 Let $I$ be a square-free anti-commutativity ideal. 
 Then $Z(KQ/I)$ does not contain
 any nilpotent elements.
\end{prop}

\section{Finite generation of the center}

Since we are interested in finite generation of the Hochschild cohomology ring, we are also interested in
when the center is finitely generated as a $K$-algebra. We'll see a nice decompostiton of the center and 
useful a necessary condition for finite generation. By introducing an enhanced generator graph we'll get a 
combinatorial tool to determine when the center is finitely generated. 


\begin{lemma}\label{sq-cen}
 Let $I$ be a square-free commutativity ideal and $a\in Q_1$. 
 If $a^k\in Z(KQ/I)$, then $a\in Z(KQ/I)$.  
\end{lemma}
\begin{proof}
 Assume $a^k\in Z(KQ/I)$. Theorem \ref{vsp} gives that for any other arrow $b\in Q_1$ either $ab- ba\in I$
 or $ab=ba=0$, i.e. $a\in Z(KQ/I)$. 
\end{proof}

\begin{prop}\label{deg1}
Let $I$ be a square-free commutativity ideal. 
If the center of $KQ/I$ is non-trivial and finitely generated as an $K$-algebra, 
then $Z^+(KQ/I)$ is generated in degree 1. Conversely, if $Z^+(KQ/I)$ is generated in degree $1$
we have that
$Z(KQ/I)$ is finitely generated and non-trivial.
\end{prop}
\begin{proof}
Assume $Z^+(KQ/I)$ is finitely generated by generators $p_1,\dots,p_n$. 
We may, by Theorem \ref{vsp}, assume each $p_i$ is a product of commuting loops at the same basepoint. 

Let $a$ be any arrow occurring in any one of the $p_i$, and suppose, without loss of generality, that 
$p_1=a^m b_1\dots b_k$ is the generator in which a power of $a$ 
occurs together with a minimal number $k$ of arrows $b_i\neq a$. We may also assume that $m$ is maximal for this
value of $k$.

Since also $a^{m+1}b_1\dots b_k$ belongs to the center, it is a product of 
at least two generators $p_i$, and by the minimality of $k$, this factorization must be of the form 
$a^j\cdot a^{m+1-j}b_1\dots b_k$. Hence some power $a^j\in Z(KQ/I)$, with $j>0$, and by the previous lemma, 
in fact $a\in Z(KQ/I)$. 

Repeating this argument for each arrow in each of the finitely many generators gives the desired conclusion.

The converse is obvious since the quiver is assumed to be finite.
\end{proof}


As the positively graded part of the center consists of paths 
of loops it can be 
decomposed into subspaces
around the vertices.
We make the following definition.
\begin{defn}\label{subquiver-ideal}
Let $Q_x$ be the subquiver
of $Q$ consisting of the point $x$, 
arrows $a$ such that $\mathfrak{o}(a)=x$ or $\mathfrak{t}(a)=x$ and the 
vertices where these arrows have their origin/target. The point $x$ is called the basepoint of $Q_x$.
 Let $I_x$ be the ideal generated by the relations $\rho$ in $I$ such that the arrows in the relation $\rho$ 
 lie in $Q_x$. 
\end{defn}
If $I$ is an admissible (anti-)commutativity ideal, then $I_x$ will be an admissible (anti-)commutativity
ideal of $KQ_x$. If $I$
is square-free, $I_x$ will be square-free. It follows from Theorem \ref{vsp} and Theorem \ref{vspsemi} that
$$Z^+(KQ/I)=\bigoplus_{x\in Q_0}Z^+(KQ_x/I_x).$$

Another way to have a finitely generated center is if 
the center is trivial, i.e. $Z(KQ/I)=K$. 

Let $M$ be a set of arrows of $Q$ such that $ab=ba\neq 0$ for all $a,b\in M$. Let $B$ be the
set of all monomials in $KQ/I$ formed from $M$. 
We call the sets
$B_1, B_2, \dots B_k$ the \tbf{commutating blocks} of $KQ/I$. Note that these 
sets $B_1, B_2, \dots, B_k$
don't have to be disjoint, since commutativity is not a transitive property in partly
commutative quiver algebras. 
One can also find these blocks from the generator graph $\Gamma_I$ by first removing all the
directed edges and then considering the cliques in the remaining graph. The cliques correspond
to the commutating blocks.
In an analogous way we define the \tbf{anti-commutating blocks}.

\begin{prop}
 Let $I$ be a square-free (anti-)commutativity ideal. Let $\La=KQ/I$ and let 
 $\La_x=KQ_x/I_x$. Then $Z(\La_x)$ is non-trivial if and only if there exists an (anti-)commutating block, $B$,
 of $KQ_x/I_x$ such that both of the following conditions hold:
 \begin{itemize}
  \item[(i)] For any arrow $c\notin B$ with $\ori(c)=x$ there exists an arrow $a\in B$ such that $ac\in I_x$.
  \item[(ii)] For for any arrow $d\notin B$ with $\tar(d)=x$ there exists an arrow $b\in B$ such that $db\in I_x$.  
 \end{itemize}
\end{prop}
\begin{proof}
 By Theorem \ref{vsp} and Theorem \ref{vspsemi} we have that if $a_1a_2\dots a_n\in Z(\La)$ we have that
 all $a_i$ belongs to the same (anti-)commutating block and for all other arrows with origin or target $x$
 we have the stated property above.
 
 Conversely, if the center is trivial, no such block exists.
\end{proof}

If $a$ and $b$ are loops at the same basepoint with $a\in Z(KQ/I)$, $a$ and $b$ may in general commute either
trivially, i.e. $ab=ba=0$, or non-trivially, i.e. $ab=ba\neq 0$. However, when $I^\perp$ is admissible,
only the latter case can happen:
\begin{lemma}\label{lemma-fg-cen}
Suppose $I$ is a commutativity ideal such that $I^\perp$ is admissible.
If $a$ is a loop such that $a\in Z(KQ/I)$ then 
\begin{itemize}
 \item[(i)] for any other loop $b$ with $\ori(b)=\ori(a)$ we have that $ab=ba\neq 0$,
 \item[(ii)] for every arrow $c$ such that $\ori(c)\neq\tar(c)=\ori(a)$ we have that $ca=0$ and
 \item[(iii)]for every arrow $d$ such that $\tar(d)\neq\ori(d)=\ori(a)$ we have that $ad=0$.
\end{itemize}
\end{lemma}
\begin{proof}
 Assume that $a\in Z(KQ/I)$. If there exists a loop $b$ with $\ori(b)=\ori(a)$ such that $ab-ba$
 is not a commutation relation of $I$, 
then, since $ab=ba$ and composition of two loops at the same basepoint is never trivially $0$, 
we have that $ab\in I$ and $ba\in I$. If $ab\in I$ and $ba\in I$
we have a contradiction to the assumption that $I^\perp$ is 
admissible since $\Gamma_I$ will then contain a directed cycle and by Theorem \ref{adm} $I^\perp$
is then not admissible. 

For (ii), since $\ori(c)\neq\tar(a)$, we have that $ac=0$ and hence $ca=0$. 
An analogous argument in (iii) gives that $ad=0$.  
\end{proof}

\begin{prop}\label{fg-cen-prop}
  Let $I$ be a commutativity ideal such that $I^\perp$ is admissible. Let $\La=KQ/I$ and let 
 $\La_x=KQ_x/I_x$. 
 If $Z(\La_x)$ is non-trivial and finitely generated we have that there exists a non-empty set 
 of loops, $a$, such that $a$ commutes non-trivially with all other loops with basepoint $x$ and 
 for any $c$ such that $\ori(c)=x\neq\tar(c)$ we 
 have $ac\in I_x$
 and for any $d$ such that $\tar(d)=x\neq\ori(d)$ we have $da\in I_x$. 
\end{prop}
\begin{proof}
 It follows from Lemma \ref{lemma-fg-cen} that any $a\in Z(KQ/I)$ has to commute non-trivially with any
 other loop at the same basepoint. The same lemma also gives the relations to other arrows in $Q_x$. 
\end{proof}

For $Z(KQ/I)$ we sum up the necessary condition for finite generation in the following theorem.
\begin{thm}\label{fg-cen-thm}
  Suppose $I$ is a commutativity ideal such that $I^\perp$ is admissible. If $Z(KQ/I)$ is finitely
  generated as a $K$-algebra then for all $x$ in $Q_0$ either $Z(KQ_x/I_x)$ is trivial
  or there exists a non-empty set, $S$, of arrows $a$, such that
  $$I_x\supseteq \langle ab-ba, ca, ad
 \rangle_{\tiny{\begin{array}{l}a\in S,\\ b \textrm{ loop with basepoint }x,\\ c,d\textrm{ arrows such that }\ori(c)\neq\tar(c)=x, \tar(d)\neq\ori(d)=x
           \end{array}}}
.$$
\end{thm}
\begin{proof}
 Follows directly from the vector space decomposition of $Z^+(KQ/I)$ and Proposition \ref{fg-cen-prop}.  
\end{proof}


That the conditions stated in the preceding proposition are not sufficient can be seen in the following example.
\begin{exmp}\label{counterexample}
 Let $Q$ be the following quiver
  $$
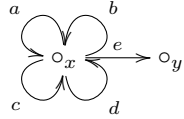
\xymatrix{\circ_x \ar@(u,l)[]_{a} \ar@(r,u)[]_{b} 
  \ar@(d,r)[]_{d} 
  \ar@(l,d)[]_{c} \ar[r]^e & \circ_y 
  }$$
  and
  $I=\langle ab-ba, ac-ca, ad-da, cd-dc, bc, db, ae, ce  \rangle$.
  Refering to Theorem \ref{fg-cen-thm}, $S=\{a\}$, since 
  $$I_x\supseteq \{ab-ba, ac-ca, ad-da,ae\}.$$ 
  But we also have that $cd\in Z(KQ/I)$ and moreover $Z(KQ/I)$ will contain $c^md^n$ for $m,n\ge 1$, 
  but not $c^m$ or $d^n$ for any $m,n$. Hence $Z(KQ/I)$ is infinitely generated by Proposition \ref{deg1}.
  
\end{exmp}

To be able to see if the center is finitely generated we need to construct a new graph.
The relation graph is the generator graph enhanced with arrows for all trivial relations. 
\begin{defn}
 Let $Q$ be a quiver and $I$ a (anti-)commutativity ideal. We define the 
 \tbf{relation graph}, $\Gamma_\mrm{rel}$, of $KQ/I$ as follows:
 \begin{itemize}
  \item For every arrow $a\in Q_1$ we have a vertex $a\in \Gamma_\mrm{rel}$.
  \item For every pair of arrows $a,b\in Q_1$, such that $ab=0\in KQ/I$ we have a directed edge
  $a\to b$ in $\Gamma_\mrm{rel}$.
  \item For every (anti-)commutativity relation $ab-ba\in I_2$ or $ab+ba\in I_2$ we have an 
undirected edge between $a$ and $b$.

By a clique in $\Gamma_\mrm{rel}$ we mean an undirected subgraph of $\Gamma_\mrm{rel}$ that is a complete graph.
We denote a clique by its vertex set. 
 \end{itemize}
\end{defn}
For our
purposes we can consider $KQ_x/I_x$, for all $x$, to find out everything we need about $Z(KQ/I)$.

The following theorem gives one of our main results, namely necessary and sufficient conditions
for finite generation of the center. 
\begin{thm}\label{rel-graph}
 Let $I$ be a commutativity ideal such that $I^\perp$ is admissible. Consider the relation graph 
 $\Gamma_\mrm{rel}$. 
 \begin{itemize}
  \item[(i)] We have that $a_1a_2\dots a_n\in Z(KQ/I)$ if and only if $\{a_1, \dots, a_n\}$ is 
  a clique of loops in $\Gamma_\mrm{rel}$ and for any other vertex $b$
  in the graph either 
  \begin{itemize}
  \item there exists a directed edge from some $a_i$ in the clique to $b$ and a directed
  edge from $b$ to some $a_j$ in the clique \textbf{or}
  
  \item $\{a_1, \dots, a_n, b\}$ is also a clique in $\Gamma_\mrm{rel}$. 
 \end{itemize}
    \item[(ii)] $Z(KQ/I)$ is finitely generated if and only if whenever we have a clique fulfilling the
conditions in (i), each of its vertices fulfills these conditions separately, considered as one-element cliques.  
 \end{itemize}
\end{thm}
\begin{proof}
 Part (i) follows directly from Theorem \ref{vsp}. Part (ii) follows from Proposition \ref{deg1} and Lemma \ref{lemma-fg-cen}.  
\end{proof}

\begin{exmp}\label{counterexample-forts}
 The relation graph for the algebra in Example \ref{counterexample} looks as follows:
 
 $$
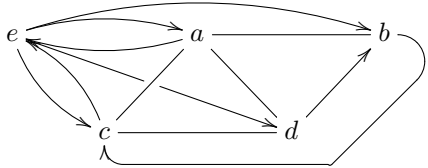
\xymatrix{e \ar@/^/[rr] \ar@/^1pc/[rrrr] \ar@/_/[dr] \ar[drrr]|!{[d];[rrr]}\hole &&a \ar@/^/[ll] \ar@{-}[rr]&
 &b \ar `_dl[dl]`/1pt[ll]`l/7pt[lll]`_u \\
 &c \ar@/_/[ul] \ar@{-}[rr] \ar@{-}[ur]&&d \ar[ur] \ar@{-}[ul]&
  }$$
 
 From this graph we can see that the clique $\{a, c, d\}$ fulfills property (i)
 in Theorem \ref{rel-graph}, while neither $c$ nor $d$ fulfills this property on their own.
 By Theorem \ref{rel-graph} $Z(KQ/I)$ is infinitely generated.
 
 One may also consider the clique $\{a,b\}$. This clique does not fulfill property (i) in Theorem 
 \ref{rel-graph},
 since for vertex $c$ we have only a directed edge from the clique to $c$ and none from $c$ to the clique. 
 Moreover, there is a directed edge from $\{a,b\}$ to $d$, but none in the reverse direction.
\end{exmp}

%

 
 There is, of course, an anti-commutative analogue of Proposition \ref{fg-cen-prop}, Theorem \ref{fg-cen-thm}
 and Theorem \ref{rel-graph}.
 With the help of the two following lemmas 
 we get the results for partly anti-commutative quiver algebras with analogous proofs.
 
\begin{lemma}\label{sq-cen-anti}
 Let $I$ be a square-free anti-commutativity ideal and $a\in Q_1$. 
 Assume $a^{2k}\in Z(KQ/I)$, then $a^2\in Z(KQ/I)$.  
\end{lemma}
\begin{proof}
 Assume $a^{2k}\in Z(KQ/I)$. Theorem \ref{vspsemi} gives that for any other arrow $b\in Q_1$ 
 either $ab+ ba\in I$
 or $ab=ba=0$, i.e. $a^2\in Z(KQ/I)$. 
\end{proof}

 \begin{lemma}\label{deg2}
Let $I$ be a square-free anti-commutativity ideal. 
If $Z(KQ/I)$ is finitely generated as a $K$-algebra, 
then $Z^+(KQ/I)$ is generated by elements
of the form $a^2$. 

\end{lemma}
\begin{proof} 
Assume $Z(KQ/I)$ is finitely generated by generators $p_1,\dots,p_n$. 
We may, by Theorem \ref{vspsemi}, assume each $p_i$ is a product of anti-commuting loops at the same 
basepoint. 

Let $a$ be any arrow occurring in any one of the $p_i$, and suppose, without loss of generality, 
that $p_1=a^m b_1\dots b_k$ is the generator in which a power of $a$ 
occurs together with a minimal number $k$ of arrows $b_i\neq a$.

Since also $a^{m+2}b_1\dots b_k$ belongs to the center, it is a product of 
at least two generators $p_i$, and by the minimality of $k$, this factorization must be of the 
form $a^j\cdot a^{m+2-j}b_1\dots b_k$. Hence some power $a^j\in Z(KQ/I)$, with $j>0$, 
and by the previous lemma, in fact $a^2\in Z(KQ/I)$. 

Repeating this argument for each arrow in each of the finitely many generators gives the desired conclusion.
\end{proof}

\begin{lemma}\label{lemma-semi-onepoint}
Suppose $I$ is a anti-commutativity ideal such that $I^\perp$ is admissible.
If $a$ is a loop such that $a^2\in Z(KQ/I)$ then 
\begin{itemize}
 \item[(i)] for any other loop $b$ with $\ori(b)=\ori(a)$ we have that $ab=-ba\neq 0$,
 \item[(ii)] for every arrow $c$ such that $\ori(c)\neq\tar(c)=\ori(a)$ we have that $ca=0$ and
 \item[(iii)]for every arrow $d$ such that $\tar(d)\neq\ori(d)=\ori(a)$ we have that $ad=0$.
\end{itemize}
\end{lemma}
\begin{proof}
Completely analogous to the proof of Lemma \ref{lemma-fg-cen}. 
\end{proof}

\begin{prop}\label{fg-cen-prop-anti}
  Let $I$ be n anti-commutativity ideal such that $I^\perp$ is admissible. Let $\La=KQ/I$ and let 
 $\La_x=KQ_x/I_x$. 
 If $Z(\La_x)$ is non-trivial and finitely generated we have that there exist a non-empty set 
 of loops, $a$, such that $a$ anti-commutes, non-trivially, with all other loops with basepoint $x$ and 
 any such loop has the property that for any $c$ such that $\ori(c)=x\neq\tar(c)$ we 
 have $ac\in I_x$
 and for any $d$ such that $\tar(d)=x\neq\ori(d)$ we have $da\in I_x$. 
\end{prop}

\begin{thm}\label{fg-cen-thm-anti}
  Suppose $I$ is an anti-commutativity ideal such that $I^\perp$ is admissible. If $Z(KQ/I)$ is finitely
  generated as a $K$-algebra then for all $x$ in $Q_0$ either $Z(KQ_x/I_x)$ is trivial
  or there exists a non-empty set, $S$, of loops, such that
  $$I_x\supseteq \langle ab+ba, ca, ad
 \rangle_{\tiny{\begin{array}{l}a\in S,\\b\textrm{ loop with basepoint }x,\\ c,d\textrm{ arrows such that }\ori(c)\neq\tar(c)=x, \tar(d)\neq\ori(d)=x
           \end{array}}}
.$$
\end{thm}

Picking the central monomials in more subtle in the anti-commutative case, and we refer to Theorem \ref{vspsemi}
for the exact procedure. However, finite generation of the center can be read off from the relation graph
just as in the commutative case. 
\begin{thm}\label{rel-graph-anti}
 Let $I$ be an anti-commutativity ideal such that $I^\perp$ is admissible. 
 Consider the relation graph $\Gamma_\mrm{rel}$. 
 \begin{itemize}
  \item[(i)] If $a_1a_2\dots a_n\in Z(KQ/I)$ then the vertices $a_1, \dots, a_n$ are all
  contained in a clique in $\Gamma_\mrm{rel}$ and for any other vertex $b$
  in the graph either 
  \begin{itemize}
  \item there exists a directed edge from some $a_i$ in the clique to $b$ and a directed
  edge from $b$ to some $a_j$ in the clique \textbf{or}
  
  \item $\{a_1, \dots, a_n, b\}$ is also a clique in $\Gamma_\mrm{rel}$. 
 \end{itemize}
    \item[(ii)] $Z(KQ/I)$ is finitely generated if and only if whenever we have a clique fulfilling
the conditions in (i), each of its vertices fulfills these conditions separately, considered as one-elements cliques.  
 \end{itemize}
\end{thm}


\section{The graded center}\label{the_graded_center}

In this section we give som lemmas relating the graded center to the even-degree center $Z^\ev(KQ/I)$,
which will be needed in the next section. We assume that $\mathrm{char}(K)\neq 2$.

Since quiver algebras bound by homogeneous ideals are graded algebras the center $Z(KQ/I)$ also has a grading.
Hence $Z(KQ/I)=\bigoplus_{k\ge0}Z^k(KQ/I)$. Let $$Z^\ev(KQ/I):=\bigoplus_{k\ge 0}Z^{2k}(KQ/I),$$ i.e. the
parts of the center of even grades. Obviously $Z^\ev(\La)\subseteq Z(\La)$.

\begin{defn}
 Let the degree of a monomial $p$ in $KQ/I$ be denoted by $|p|$. 
 The \tbf{graded center} $Z_\gr(KQ/I)$ 
 consists of the elements $p\in KQ/I$ such that
 $pq=(-1)^{|p||q|}qp$ for all $q\in KQ/I$.
\end{defn}

The graded center is of course graded and $Z_\gr^0(\La)=K$. Let $Z_\gr^+(\La)=\bigoplus_{k\ge 1}Z^k_\gr(\La)$.

\begin{lemma}\label{graded-monomial}
 Let $I$ be a square-free (anti-)commutativity ideal and $\La=KQ/I$. 
 Then $Z_{\gr}^+(\La)$ is spanned by non-zero monomials $a_1a_2\dots a_n$ of loops at the
 same basepoint such that
 $(a_ia_j)$ is an allowed transposition for any pair $a_i,a_j$ in the monomial.  
\end{lemma}
\begin{proof}
 Analogous to the proof of Lemma \ref{sums} we get 
 that $Z_{\gr}^+(\La)$ is generated by permutation homogeneous elements.
 Assume $\sum_i p_i\in Z_{\gr}(\La)$ is a permutation homogeneous element.
 Then, for any $a\in Q_1$, we have $\sum_i ap_i=\sum_i (-1)^{|p_i|}p_ia$. 
 Since $\sum_p p_i$ is permutation homogeneous we have, for all $p_i$ in the sum, that $|p_i|=c$
 for a fixed number $c\in \mathbb{N}$. Assume $p_k=a_1a_2\dots a_n$ and $a_1p_k\neq 0$, then there exists 
 an $p_l$ such that $a_1p_k=(-1)^jp_la_1$ and by Lemma \ref{comm-bin} and Lemma \ref{anti-comm-sign}
 $p_k\sim p_l$, i.e. $(a_1a_m)$ is an allowed transposition for all $2\le m\le n$. Inductively
 we get that $(a_ia_j)$ are allowed transpositions for all $1\le i,j\le n$ (see proof of
 Theorem \ref{mono-center} for the technique). Hence $p_k=p_i$ for all $p_i$ in the sum
 $\sum_i p_i$ and the graded center is generated by monomials.
 \end{proof}

\begin{prop}\label{graded-center}
 Let $I$ be a square-free commutativity ideal. Then $Z_{\gr}(\La)=Z^{\ev}(\La)$.
\end{prop}
\begin{proof}
 It is clear that $Z^{\ev}(\La)\subseteq Z_{\gr}(\La)$, since if $|p|=2k$ and $p\in Z(\La)$ 
 we have that 
 $pq=(-1)^{(2k)|q|}qp=qp$ for all $q\in \La$.
 
 Assume that $p\neq 0$ is a monomial, $p\in Z_{\gr}(\La)$ and $|p|$ odd. Then for any $q\in \La$
 we have that $pq=(-1)^{|p||q|}qp$. Hence $p^2=-p^2$. Since $p\in Z(\La)$ and $I$ square-free
 we have that $p^2\neq 0$ which gives a contradiction and hence
 $Z_\gr(\La)$ does not contain any elements of odd degree. 
\end{proof}

\begin{prop}\label{graded-anti}
 Let $I$ be a square-free anti-commutativity ideal. Then $Z_{\gr}(\La)=Z^\ev(\La)$.
\end{prop}
\begin{proof}
 It is clear that $Z^{\ev}(\La)\subseteq Z_{\gr}(\La)$ (see proof of the previous lemma for details). 
 Assume $p\in Z_{\gr}(\La)$ and $|p|$ odd. 
 Since $p\in Z_{\gr}(\La)$ we have $bp=-pb$ for all $b\in Q_1$. 
 Let $p=a_1a_2 \dots a_n$.
 We have that $I$ is square-free, and hence, by Lemma \ref{graded-monomial} $a_ip\neq 0$
 for any $1\le i\le n$. Since $|p|$ is odd we can find an $a_i$ such that $p$ contains an odd 
 number of copies of $a_i$.
 Since $p$ is assumed to be in the center we have $a_1p=-pa_i$ which by Lemma \ref{anti-comm-sign} 
 implies that the number of transpositions needed to transform
 $a_ip$ to $pa_i$ is odd, and hence we need an even number of copies of $a_i$ in the monomial $p$.
 This gives a contradiction and hence we have no elements of odd degree in $Z_\gr(\La)$. 
  \end{proof}

\begin{lemma}\label{even-center}
 Let $I$ be a commutativity ideal such that $I^\perp$ is admissible.
 Then $Z^\ev(\La)$ is finitely generated as a $K$-algebra if and only if $Z(\La)$ is finitely generated.
\end{lemma}
\begin{proof}
 By Proposition \ref{deg1}, if $Z(\La)$ is finitely generated, then it is generated in degree $1$. A finite 
 number of degree $1$ generators gives a finite number of degree $2$ monomials, and hence $Z^{\ev}(\La)$ is
 finitely generated if $Z(\La)$ is finitely generated. 
 
 Assume $Z^{\ev}(\La)$ finitely generated. By the same technique as in the proof of Proposition \ref{deg1} 
 we get that it is generated
 in degree $2$. 
 Assume $ab$ is one of the generators and that $a\neq b$. Assume $a\notin Z(\La)$, then $a^2\notin Z(\La)$, 
 since that would contradict Lemma \ref{sq-cen}. By Theorem 
 \ref{vsp} we have that $a^{2k+1}b\in Z^\ev(\La)$ for all $k\ge 0$. Since $a^{2k+1}b$ can't be obtained as
 a product of elements in degree $2$, unless $a^2$ is one of the generators, we get a contradiction
 to our assumption that $Z^\ev(\La)$ is finitely generated. 
\end{proof}

\begin{lemma}\label{anti-even-center}
 Let $I$ be a anti-commutativity ideal such that $I^\perp$ is admissible.
 Then $Z^\ev(\La)$ is finitely generated as a $K$-algebra if and only if $Z(\La)$ is finitely generated. 
 Moreover
 $Z^{\ev}(\La)=Z(\La)$. 
\end{lemma}
\begin{proof}
By Lemma \ref{deg2} we have that if $Z(\La)$ is finitely generated, 
then it is generated in degree $2$ and hence $Z(\La)=Z^{\ev}(\La)$, i.e.
if $Z(\La)$ is finitely generated then so is $Z^{\ev}(\La)$.

Assume $Z^\ev(\La)$ is finitely generated. By the same technique as in the proof of Lemma \ref{deg2} 
 we get that it is generated in degree $2$. 
 Assume that $ab\in Z^\ev(\La)$, then Theorem \ref{vspsemi}
 gives that $a=b$ (since if $a\neq b$, then $ab\notin Z(\La)$) and hence $Z(\La)$ is finitely 
 generated and $Z(\La)=Z^\ev(\La)$. 
\end{proof}

\begin{lemma}\label{even-nilpotent}
 Let $I$ be an (anti-)commutativity ideal such that $I^\perp$ is admissible. 
 Then $Z^{ev}(\La)$ does not contain any nilpotent elements.
\end{lemma}
\begin{proof}
 By Lemma \ref{square-free-admissible} $I$ is a square-free ideal and hence
 the result follows directly from Proposition \ref{nilpotent} (and Proposition \ref{nilpotent-anti}).
\end{proof}

\section{Finite generation of the Hochschild cohomology ring}

Let $\La=KQ/I$ be a quiver algebra for an admissible ideal $I$. If $\La$ is a Koszul algebra, its Koszul dual
is defined by $\La^!:=KQ^{\op}/I_\mrm{o}^\perp$, where $Q^\op$ is the quiver such that $Q^\op_0=Q_0$ and 
if $a\in Q_1$, with $\ori(a)=x$ and $\tar(a)=y$, then $a^\mrm{o}\in Q^\op_1$ with $\ori(a^\mrm{o})=y$ 
and $\tar(a^\mrm{o})=x$. If $p=a_1a_2\dots a_n\in KQ$ we let $p^\mrm{o}=a_n^\mrm{o}a_{n-1}^\mrm{o}\dots a_1^\mrm{o}\in KQ^\op$.
 If $I$ is an ideal of $KQ$, let $I_\mrm{o}$ be the ideal of $KQ^\op$ such that 
 $\sum_{i=1}^n p_i\in I$ if and only if $\sum_{i=1}^n p_i^\mrm{o}\in I_\mrm{o}$.
 
 If $\La$ is partly commutative, then $\La^!$ will be partly anti-commutative. 
By \cite{buchweitz2008multiplicative} (Theorem 4.1) and \cite{green1998koszul} (Theorem 2.2) 
there is an isomorphism $\HH(\La)/\Ni\cong Z_\gr(\La^!)/\Ni_Z$, provided that $\La$ is a Koszul algebra.
In general it is very hard to tell if an algebra is Koszul or not, 
but there are some examples of partly (anti-)commutative quiver algebras that are Koszul algebras. The
theorems in this chapter give necessary and sufficient conditions for a partly (anti-)commutative Koszul quiver
algebra to have finitely generated Hochschild cohomology ring modulo nilpotent elements.

 To use the sufficient conditions from Theorem \ref{rel-graph} and Theorem \ref{rel-graph-anti} 
 we need the following theorem.
\begin{thm}\label{hochschild}
 Let $I$ be an (anti-)commutativity ideal and $\La=KQ/I$ a Koszul algebra. The 
 Hochschild cohomology ring of $\La$ modulo nilpotence, $\HH(\La)/\Ni$, is finitely generated if and only if
 $Z(\La^!)$ is finitely generated. 
\end{thm}
\begin{proof}
 We have that $\HH(\La)/\Ni=Z_\gr(\La^!)/\Ni_Z$ and by Proposition \ref{graded-center}/\ref{graded-anti}
 and Lemma \ref{even-nilpotent} we have that 
 $Z_\gr(\La^!)/\Ni_Z=Z^\ev(\La^!)$. By Lemma \ref{even-center}/\ref{anti-even-center} we have that
 $Z^\ev(\La^!)$ is finitely generated if and only if $Z(\La^!)$ is finitely generated.  
\end{proof}

\begin{exmp}
\begin{itemize}
 \item[(i)]
 Let $\La=KQ/I$ be the algebra described in Example \ref{ex1}(i) and Example \ref{cen-ex}, 
 and let $x$ be the left and $y$
 the right vertex. 
 This algebra is a Koszul algebra (see \cite{snashall2008support})
 and we have that $Z((\La^!)_y)$ is trivial. For $Z((\La^!)_x)$ we draw the relation graph of 
 $KQ^\mrm{op}_x/(I^\perp_\mrm{o})_x$:
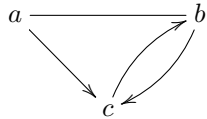
 $$\xymatrix{ a \ar@{-}[rr] \ar[dr] & & b\ar@/^/[dl]\\& c  \ar@/^/[ur] &
 }$$
 
 We see that $Z((\La^!)_x)$ is not finitely generated since we don't have any directed edge from
 $a$ to $c$. 
 Hence $\HH(\La)$ is not finitely generated (this is also proven in \cite{snashall2008support}). 
 
 \item[(ii)]
 Let $\La=KQ/I$ be the algebra described in Example \ref{ex1}(ii), and let $x$ be the left and $y$
 the right vertex. 
 The algebra described in Example \ref{ex1}(ii) is a Koszul algebra since it is a quadratic, monomial 
 algebra (see \cite{monomial}).
 
 Both $Z(\La_x)$ and $Z(\La_y)$ are trivial. This can be realized in two ways, either we use Theorem
 \ref{vsp} and see that the only potential elements in the center are of the form $a^mb^m$, but that 
 requires $a$ and $b$ to commute, and since they don't, none of these can be in the center. The other 
 way is to draw the relation graph of $KQ^\mrm{op}_x/(I^\perp_\mrm{o})_x$.
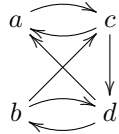
 $$\xymatrix{ a \ar@/^/[r] & c \ar@/^/[l] \ar[d] \\b \ar[ur] \ar@/^/[r] & d \ar@/^/[l] \ar[ul]
 }$$
 Obviously none of the cliques of loops, $\{a\}$ and $\{b\}$, have directed edges to and from all
 other vertices or are contained in bigger cliques. Hence $\HH(\La)$ is trivial.
 
 \item[(iii)]Let $Q$ be the quiver in Example \ref{ex1}(iii). Let $I$ be the admissible ideal generated by
 $\langle a^2, b^2, c^2, d^2, cd, dc, ab+ba, ac+ca, ad+da, bc+cb, bd+db\rangle$. Then $KQ/I$ is a 
 Koszul algebra. The relation graph of $KQ^\op/I^\perp_\mrm{o}$ looks as follows:
 $$
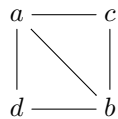
\xymatrix{a \ar@{-}[r] \ar@{-}[d] \ar@{-}[dr] & c \ar@{-}[d]\\
 d \ar@{-}[r]&b
 }$$
 Here, the clique $\{a,b\}$ have the property that if we adjoin any of the other vertices we get a 
 bigger clique and both $a$ and $b$ have this property on their own. Hence 
 $Z(\La^!)$ is generated by $a$ and $b$, and   
 we have that $\HH(\La)$ is finitely generated.

\end{itemize}

\end{exmp}

{\footnotesize\bibliography{references}{}}

\begin{thebibliography}{BGSS08}

\bibitem[BGSS08]{buchweitz2008multiplicative}
Ragnar-Olaf Buchweitz, Edward~L Green, Nicole Snashall, and {\O}yvind Solberg.
\newblock Multiplicative structures for {K}oszul algebras.
\newblock {\em The Quarterly Journal of Mathematics}, 59(4):441--454, 2008.

\bibitem[Eve61]{evens1961cohomology}
L.~Evens.
\newblock The cohomology ring of a finite group.
\newblock {\em Transactions of the American Mathematical Society},
  101(2):224--239, 1961.

\bibitem[FS97]{friedlander1997cohomology}
E.M. Friedlander and A.~Suslin.
\newblock Cohomology of finite group schemes over a field.
\newblock {\em Inventiones Mathematicae}, 127(2):209--270, 1997.

\bibitem[GMV98]{green1998koszul}
E.L. Green and R.~Mart{\'\i}nez-Villa.
\newblock Koszul and {Y}oneda algebras {II}.
\newblock {\em Algebras and modules, II (Geiranger, 1996}, 24:227--244, 1998.

\bibitem[GSS03]{green2003hochschild}
E.L. Green, N.~Snashall, and {\O}.~Solberg.
\newblock The {H}ochschild cohomology ring of a selfinjective algebra of finite
  representation type.
\newblock {\em Proceedings of the American Mathematical Society},
  131(11):3387--3394, 2003.

\bibitem[GSS04]{green2004hochschild}
E.L. Green, N.~Snashall, and {\O}.~Solberg.
\newblock The {H}ochschild cohomology ring modulo nilpotence of a monomial
  algebra.
\newblock {\em Arxiv preprint math/0401446}, 2004.

\bibitem[GZ94]{monomial}
E.~Green and D.~Zacharia.
\newblock The cohomology ring of a monomial algebra.
\newblock {\em manuscripta mathematica}, 85:11--23, 1994.
\newblock 10.1007/BF02568180.

\bibitem[Hap89]{happel1989hochschild}
D.~Happel.
\newblock Hochschild cohomology of finite-dimensional algebras.
\newblock {\em S{\'e}minaire d’Algebre Paul Dubreil et Marie-Paul Malliavin},
  pages 108--126, 1989.

\bibitem[PS11]{parker2011family}
A.~Parker and N.~Snashall.
\newblock A family of {K}oszul self-injective algebras with finite {H}ochschild
  cohomology.
\newblock {\em Journal of Pure and Applied Algebra}, 2011.

\bibitem[Sna08]{snashall2008support}
N.~Snashall.
\newblock Support varieties and the {H}ochschild cohomology ring modulo
  nilpotence.
\newblock {\em Arxiv preprint arXiv:0811.4506}, 2008.

\bibitem[SS04]{snashall2004support}
N.~Snashall and {\O}.~Solberg.
\newblock Support varieties and {H}ochschild cohomology rings.
\newblock {\em Proceedings of the London Mathematical Society}, 88(3):705--732,
  2004.

\bibitem[Xu08]{xu2008hochschild}
F.~Xu.
\newblock Hochschild and ordinary cohomology rings of small categories.
\newblock {\em Advances in Mathematics}, 219(6):1872--1893, 2008.

\bibitem[XZ11]{xu2011more}
Y.~Xu and C.~Zhang.
\newblock More counterexamples to {H}appel's question and
  {S}nashall-{S}olberg's conjecture.
\newblock {\em Arxiv preprint arXiv:1109.3956}, 2011.

\end{thebibliography}
\bibliographystyle{alpha}
\end{document}